\documentclass[12pt]{amsart}

\usepackage{parskip}

\makeatletter
\def\thm@space@setup{%
  \thm@preskip=\parskip \thm@postskip=0pt
}
\makeatother

\usepackage[dvipsnames]{xcolor}
\usepackage{mathtools}
\usepackage{graphicx}
\usepackage{amsmath}
\usepackage{color}
\usepackage{caption}
\usepackage{subcaption}
\usepackage{mwe}
\usepackage{amsfonts,amssymb,amscd}
\usepackage{mathrsfs}

\calclayout

\newtheoremstyle{remboldstyle}
  {}{}{\itshape}{}{\bfseries}{.}{.5em}{{\thmname{#1 }}{\thmnumber{#2}}{\thmnote{ (#3)}}}
\theoremstyle{remboldstyle}

\usepackage[outdir=./]{epstopdf}

\newtheorem{thm}{Theorem}[section]
\newtheorem{prop}[thm]{Proposition}
\newtheorem{lem}[thm]{Lemma}

\newtheorem{conj}[thm]{Conjecture}

\theoremstyle{definition}
\newtheorem{definition}[thm]{Definition}

\newtheorem{rem}[thm]{Remark}
\newtheorem{construction}[thm]{Construction}
\newtheorem{notation}[thm]{Notation}

\numberwithin{equation}{section}

\newcommand{\C}{\mathbb{C}}  

\newcommand{\ol}{\overline}

\newcommand{\Chat}{\widehat{\mathbb{C}}}

\begin{document}


\title[logharmonic polynomials]{Sharp bounds for the valence of certain logharmonic polynomials}


\author{Kirill Lazebnik}
\thanks{\noindent The first author is supported in part by NSF grant DMS-2452130. } 


\author{Erik Lundberg}
\thanks{\noindent The second author is supported in part by Simons Foundation grant 712397.} 




\begin{abstract}
Consider a logharmonic polynomial; that is, a product of the form $p(z)\overline{q(z)}$, where $p$, $q$ are holomorphic polynomials. Assume $q$ is linear and denote by $n$ the degree of $p$. It was recently shown in \cite{KLP} that the valence of such a logharmonic polynomial is at most $3n-1$; in this paper we show that their $3n-1$ upper bound is sharp. Together with the work of \cite{KLP}, this resolves a conjecture in \cite{BshoutyHengartner}.
\end{abstract}


\maketitle


\section{Introduction}

In recent decades, there has been increasing interest in understanding the valence (number of preimages of a prescribed complex value) of polyanalytic functions; we refer to \cite[Introduction]{KLP} for an overview of this topic, including both the algebraic and complex analytic points of view.  Particular attention has been placed on the problem of determining the maximal valence over special classes, such as harmonic polynomials (sums of analytic and anti-analytic polynomials) or logharmonic polynomials (products of analytic and anti-analytic polynomials) of prescribed degree.

In the current paper, we are concerned with the latter setting of logharmonic polynomials.  The maximal valence problem for logharmonic polynomials asks for a sharp upper bound on the number of solutions to the equation $p(z)\overline{q(z)} = w$ in terms of $n:=\deg p$ and $m:=\deg q $, under the natural condition that $p$ is not a constant multiple of $q$ (which ensures the valence is finite \cite[Sec. 3]{AbdulhadiHengartner}).  This problem was posed by W. Hengartner at the second international workshop on planar harmonic mappings at the Technion, Haifa, January 7-13, 2000 and has been restated in \cite{BshoutyHengartner},  \cite{AbdulhadiHengartner}, \cite{Problems}, and \cite{KLP}. 
In the case $m=1$ (i.e., when $q$ is linear) Bshouty and Hengartner made a specific conjecture for the maximal valence \cite{BshoutyHengartner}.
\begin{conj}[Bshouty, Hengartner, 2000]
In the case $m=1$, the maximal valence is $3n-1$.
\end{conj}

This conjecture entails two assertions, namely, (i) an upper bound of $3n-1$ for each $n\geq 1$ and (ii) sharpness of the upper bound, in other words, existence of examples attaining valence $3n-1$ for each $n \geq 1$.
The upper bound part of the conjecture was proved in \cite{KLP}.

\begin{thm}[Khavinson, Lundberg, Perry, 2024]\label{thm:KLP}
Let $p$ be a polynomial of degree $n>1$ and let $q(z)$ be linear.  For each $w \in \C$, the number of solutions of the equation $p(z)\ol{q(z)}=w$ is at most $3n-1$.
\end{thm}

We will now discuss the sharpness part of the Bshouty-Hengartner conjecture. For $n=2$, an example with $3n-1=5$ solutions was provided by Bshouty and Hengartner \cite{BshoutyHengartner}, where examples with $3n-3$ solutions were also observed.  Numerical evidence was presented in \cite{KLP} to support the case for sharpness in the cases $n=3$ and $n=4$, but leaving open a proof of sharpness for $n=3$, $n=4$ as well as general $n>4$. The purpose of the present paper is to prove the existence of such sharpness examples for all $n$, hence completing the proof of the Bshouty-Hengartner conjecture.

\begin{thm}\label{thm:sharp}
For each $n>1$ there exists a polynomial $p$ of degree $n$, a linear polynomial $q$, and a complex number $w$ such that the equation $p(z)\ol{q(z)}=w$ has exactly $3n-1$ solutions.
\end{thm}

Showing sharpness can be reformulated as a problem in anti-holomorphic dynamics, and we arrive at Theorem \ref{thm:sharp} as a consequence of the following result (see Section \ref{sec:reformulate}):

\begin{thm}\label{main_thm} For each $n>1$, there exists a degree $n$ polynomial $p(z)$ and $c\in\mathbb{C}$ so that  $z\mapsto \overline{c}+1/\overline{p(z)}$ has $n$ attracting fixed points. 
\end{thm} 

As discussed in \cite{KLP}, Hengartner's valence problem for logharmonic polynomials is reminiscent of Sheil-Small's valence problem for harmonic polynomials \cite{sheil-small_2002}.  In that vein, the Bshouty-Hengartner conjecture runs parallel to a conjecture of Wilmshurst \cite{W2} that the maximum valence of harmonic polynomials $p(z) + \overline{q(z)}$ with $q$ linear is at most $3n-2$.  The upper bound in Wilmshurst's conjecture was proven by Khavinson and Swiatek \cite{KhSw} and sharpness was shown by Lukas Geyer \cite{MR2358495}. In this paper, we form a direct connection between Wilmshurst's conjecture and the Bshouty-Hengartner conjecture by utilizing the following result of Geyer \cite{MR2358495} in our proof of Theorem \ref{main_thm}.

\begin{thm}\label{Lukas} \cite{MR2358495} For each $n>1$, there exists a degree $n$ real polynomial $p(z)$ so that $\overline{p(z)}$ has $n-1$ finite, fixed critical points. 
\end{thm}

The original proof of Theorem \ref{Lukas} given in \cite{MR2358495} relies on Thurston's topological characterization of polynomials. Different proofs of Theorem \ref{Lukas} were later given, first in \cite{2014arXiv1411.3415L} using the Krein-Milman Theorem, and later in \cite{MR4273178} using quasiconformal surgery. We mention that \cite{MR4273178} describes all polynomials which satisfy Theorem \ref{Lukas}. 

We give a short proof of Theorem \ref{main_thm} in Section \ref{short_pf} by showing that an appropriate perturbation of the polynomial in Theorem \ref{Lukas} satisfies the conclusions of Theorem \ref{main_thm}. 

\begin{rem}\label{no_pcf_rem} Geyer's examples proving sharpness in Wilmshurst's conjecture are post-critically finite; that is, each critical value has finite forward orbit. The class of post-critically finite maps plays a central role in complex dynamics. In contrast to Geyer's examples, the maps we produce in Theorem \ref{main_thm} proving sharpness in the Bshouty-Hengartner conjecture are not post-critically finite. Moreover, in fact there are no post-critically finite examples satisfying Theorem \ref{main_thm}; see Proposition \ref{no_pcf} for a precise formulation of this fact. This indicates an important difference between the setting of Wilmshurst's conjecture and that of the Bshouty-Hengartner conjecture.
\end{rem}

Nevertheless, by a more careful approach, we are able to prove the following sharper version of Theorem \ref{main_thm}. 

\begin{thm}\label{main_thm2} For each $n>1$, there exists a degree $n$ polynomial $p(z)$ and $c\in\mathbb{C}$ so that  $z\mapsto\overline{c}+1/\overline{p(z)}$ has $n$ attracting fixed points, and the $n-1$ finite critical points of $z\mapsto\overline{c}+1/\overline{p(z)}$ are fixed. 
\end{thm} 

Theorem \ref{main_thm2} is a sharper version of Theorem \ref{main_thm} in the sense that it achieves the maximal number $n-1$ of critical values having finite forward orbits (the critical point at $\infty$ of a map satisfying Theorem \ref{main_thm} must always have an infinite forward orbit: see Proposition \ref{no_pcf}). 

Our proof of either of Theorems \ref{main_thm}, \ref{main_thm2} may roughly be described as perturbing the superattracting fixed point $\infty$ of the polynomial $p$ of Theorem \ref{Lukas} to a nearby, finite point to create a linearly attracting fixed point near $\infty$: in the proof of Theorem \ref{main_thm} this is done simply by post-composing $p$ with a M\"obius transformation, and in the proof of Theorem \ref{main_thm2} this is done more carefully (so as to preserve the property of $p$ of having $n-1$ fixed critical points) using quasiconformal surgery. Quasiconformal surgery refers to a number of techniques which arose starting in the twentieth century within Geometric Function Theory, Riemann surfaces, and later in Complex Dynamics; we refer to \cite{MR3445628} for a history.

\section{Proof of Theorem \ref{main_thm}}\label{short_pf}

\begin{lem}\label{form_of_r_lem} Let $r$ be a degree $n$ rational map. Then there exists $c\in\mathbb{C}$ and a polynomial $p(z)$ so that $r(z)\equiv c+1/p(z)$ if and only if $r$ has a degree $n-1$ critical point at $\infty$ and $r(\infty)\in\mathbb{C}$. 
\end{lem}

\begin{proof} The $\implies$ direction is evident. For the $\impliedby$ direction, set $c:=r(\infty)$ and consider the rational map 
\begin{equation}\label{rearrange} p(z):= \frac{1}{r(z)-c} \textrm{ for } z\in\Chat. \end{equation}
Since $r$ is assumed to have a degree $n-1$ critical point at $\infty$, we have $r^{-1}(c)=\{\infty\}$ and hence $p^{-1}(\infty)=\{\infty\}$; in other words the only pole of the rational map $p(z)$ is at $\infty$ and hence $p(z)$ is a polynomial. The conclusion follows.
\end{proof}

\begin{notation} Let $p$ be as in Theorem \ref{Lukas}. We let 
\begin{equation} M_\delta(z):= \frac{z+\delta}{1+\delta z} \textrm{ for } \delta \in\mathbb{C},
\end{equation}
and consider the composition
\begin{equation} r_{\delta}(z):=M_\delta\circ p(z). 
\end{equation}
\end{notation}

\begin{prop}\label{fixedpointspreserved} For $\delta$ sufficiently small, the map $z\mapsto\overline{r_{\delta}(z)}$ has $n$ attracting fixed points. 
\end{prop}

\begin{proof} The map $\overline{r_0}=\overline{p}$ has $n-1$ finite attracting fixed points by Theorem \ref{Lukas}, as well as an attracting fixed point at $\infty$. This means each of these $n$ attracting fixed points is the center of a (spherical) disc, denote it by $D$, so that $\overline{r_0}(D)$ is compactly contained within $D$ (see Lemma 8.1 of \cite{MilnorCDBook}). Thus, for sufficiently small $\delta$, we have that $\overline{r_\delta}(D)$ is still compactly contained within $D$. Thus $\overline{r_\delta}$ has a fixed point (within $\overline{r_\delta}(D)$) near each of the attracting fixed points of $\overline{r_0}=\overline{p}$, and by Schwarz's lemma each of these $n$ fixed points of $\overline{r_0}$ must be attracting. 
\end{proof}

Theorem \ref{main_thm} now follows from Lemma \ref{form_of_r_lem} and Proposition \ref{fixedpointspreserved}. Indeed, for sufficiently small $\delta\not=0$ we have that $r_{\delta}(z)$ has a degree $n-1$ critical point at $\infty$ with $r_\delta(\infty)=1/\delta\in\mathbb{C}$, so that by Lemma \ref{form_of_r_lem} we have $r_\delta(z)=c+1/p_\delta(z)$ for some polynomial $p_\delta(z)$. Moreover, for $\delta\not=0$ sufficiently small, $z\mapsto \overline{r_\delta(z)}$ also has $n$ attracting fixed points by Lemma \ref{form_of_r_lem}.

\vspace{2.5mm}

\section{Proof of Theorem \ref{thm:sharp}}\label{sec:reformulate}

 Here we finish resolving the Bshouty-Hengartner conjecture by arriving at Theorem \ref{thm:sharp} as an application of Theorem \ref{main_thm} and the generalized argument principle for complex harmonic functions.

Fix $n > 1$. By Theorem~\ref{main_thm}, there exists a degree $n$ polynomial $p(z)$ and a complex number $c \in \mathbb{C}$ such that the anti-rational map
\[
z \mapsto \overline{r(z)}, \quad \text{where } r(z) := c + \frac{1}{p(z)},
\]
has exactly $n$ attracting fixed points.

The fixed point equation $\overline{r(z)} = z$ can be rearranged as $p(z)\ol{(z - \overline{c})} = 1$ which is of the desired form $p(z)\overline{q(z)} = w$, with $q$ linear.

To count the total number of solutions, we further note that solutions to $p(z)\overline{q(z)} = 1$ are zeros of the (complex) harmonic function $H(z) := z - \ol{c} - \frac{1}{\overline{p(z)}}$ (a sum of analytic and anti-analytic functions is appropriately called harmonic as each of its real and imaginary components are real harmonic functions).  Note that the Jacobian of $H$ is $$J_H(z) = |H_z|^2 - |H_{\ol{z}}|^2 = 1 - |r'(z)|^2.$$

We will use a generalization of the argument principle for harmonic functions stated below.

First recall that a zero \( z_0 \) of \( H \) is \emph{singular} if the Jacobian $J_H(z)$
vanishes at \( z_0 \). Otherwise, \( H \) is \emph{sense-preserving} at \( z_0 \) if \( J_H(z_0) > 0 \), and \emph{sense-reversing} if \( J_H(z_0) < 0 \). The \emph{order} of a sense-preserving zero is the smallest $k > 0$ for which the $k$th derivative of the analytic part is nonzero; sense-reversing zeros are assigned the order of the corresponding sense-preserving zero of \( \overline{H} \). 

Suppose \( H(z) \to \infty \) as \( z \to z_0 \) and $H$ is free of singular zeros.  Then the Jacobian of $H$ is either positive or negative throughout a small punctured neighborhood of $z_0$.  We call \( z_0 \) a \emph{sense-preserving pole} or \emph{sense-reversing pole} accordingly, and in either case the order is defined as $|\frac{1}{2\pi} \Delta_T \arg H|$,
where \( T \) is a small circle around $z_0$.

\begin{rem}\label{rem:pole}
The function \( H(z) = z - \ol{c} - \frac{1}{\overline{p(z)}} \) has poles at the zeros of \( p \), each of which is sense-reversing, and their total order (sum of orders) is \( n \).
\end{rem}

\begin{rem}\label{rem:zeros}
The total order $N_+$ of sense-preserving zeros of $ H(z) = z - \ol{c} - \frac{1}{\overline{p(z)}} $ is $ n$. Indeed the attracting fixed points of $\overline{r(z)}$ correspond to orientation-preserving zeros of $H$ since its Jacobian is $J_H(z) = 1 - |r'(z)|^2$ which is positive precisely when we have $|r'(z)|<1$.
\end{rem}

We now state the generalized argument principle \cite[Thm.~2.2]{ST}:

\begin{lem}[Generalized Argument Principle]\label{lem:argprinc}
Let \( H \) be harmonic in a region \( D \), except for finitely many poles. Let \( C \) be a Jordan curve in \( D \) avoiding the zeros and poles of \( H \), and let \( \Omega \) denote the interior of \( C \). Suppose \( H \) has no singular zeros in \( \Omega \). Let \( N_+, N_- \) be the total orders of sense-preserving and sense-reversing zeros of \( H \) in \( \Omega \), and let \( P_+, P_- \) be the total orders of sense-preserving and sense-reversing poles. Then
\begin{equation}
    \frac{1}{2\pi} \Delta_C \arg H = N_+ - N_- - (P_+ - P_-).
\end{equation}
\end{lem}

\begin{rem} In order to apply Lemma \ref{lem:argprinc}, we need to verify that $H$ is free of singular zeros. Indeed, as remarked earlier, zeros of $H$ correspond to fixed points of $z \mapsto \overline{r(z)}$, so a singular zero of $H$ corresponds to a neutral fixed point of $\overline{r(z)}$. In order to employ standard results from holomorphic dynamics, let us consider the second iterate $s(z):=\overline{r\circ \overline{r(z)}}$. A neutral fixed point of $\overline{r}$ is also a neutral fixed point of $s$. Each neutral fixed point of $s$ must either be contained in the closure of the postcritical orbit (in the case of a Cremer point), or else the neutral fixed point is contained in a Siegel disk whose boundary is contained in the closure of the postcritical orbit (see Theorem 11.17 of \cite{MilnorCDBook}). Neither of these cases can occur, since the postcritical orbit of $s$ is contained in the postcritical orbit of $r$, and the postcritical orbit of $r$ is contained in the basins of attraction for the various attracting fixed points of $r$. Thus neither $s$ nor $r$ can have a neutral fixed point, and thus $H$ can not have any singular zeros.
\end{rem}

Noting that $H(z) \sim z$ as $|z| \to \infty$, which implies that the net change in argument along a large circle $\gamma$ is $2\pi$, we apply the generalized argument principle to $H$ to obtain:
\[
1=\frac{1}{2\pi} \Delta_\gamma \arg H(z) = N_+ - N_- - (0 - n),
\]
where we have used Remark \ref{rem:pole} to determine $P_+ = 0$ and $P_- = n$.

Combining this with the key observation in Remark \ref{rem:zeros}, we have
\[
N_- = N_+ + n - 1 = 2n - 1.
\]

So the total number of solutions is $N_+ + N_- = n + (2n - 1) = 3n - 1.$ This completes the proof of Theorem~\ref{thm:sharp}.

\begin{rem}
Concerning the conclusion of Theorem \ref{thm:sharp}, providing the existence of at least one example, we note that one extremal example implies the existence of many; starting with the extremal $H(z) = z - \ol{c} - \frac{1}{\overline{p(z)}}$ from above, and letting the parameter $c$ vary, the set of $c$ for which $H$ still has $3n-1$ zeros is open as follows from \cite[Prop. 3]{BergErem2010}. 
\end{rem}


\section{Proof of Theorem \ref{main_thm2}}

Throughout this section we reserve the terminology \emph{polynomial} for a polynomial in one complex variable $z$. If $p$ is a polynomial, we call $\overline{p}$ an \emph{anti-polynomial}. Similarly, the terms \emph{rational mappings} or \emph{rational functions} will refer to rational functions of one complex variable, and if $r$ is a rational function we call $\overline{r}$ \emph{anti-rational}. 


\begin{lem}\label{branched_cover} For any degree $n$ polynomial $p$ and $c\in\mathbb{C}$, the anti-rational map $\overline{r}:=\overline{c}+1/\overline{p}$ has at most $n$ attracting fixed points.
\end{lem}

\begin{proof} Let $\overline{r}$ be as in the statement. Degree $n$ rational (or anti-rational) maps have $2n-2$ critical points (counting multiplicity), and maps of the form $\overline{r}=c+1/\overline{p}$ have a degree $n-1$ critical point at $\infty$. Thus, $\overline{r}$ has at most 
\[ 2n-2-(n-2)=n \]
distinct critical points. Since the basin of attraction of each attracting fixed point of $\overline{r}$ must contain a critical value, $\overline{r}$ has at most $n$ attracting fixed points.
\end{proof}

\begin{rem} The anti-rational map we construct in the proof of Theorem \ref{main_thm2} has $n-1$ fixed critical points, and the remaining attracting fixed point is linearly attracting and contains the critical point $\infty$ in its immediate basin of attraction. 
\end{rem}

\noindent We now turn to the proof of Theorem \ref{main_thm2}; the proof will be broken up into a number of smaller steps.

\begin{notation} For $\delta\in(-1,1)$, we let 
\[ B_\delta(z):= \frac{z^n+\delta}{1+\delta z^n}. \] 
\end{notation}

\begin{rem} The map $B_\delta$ is a Blaschke product; indeed $B_\delta$ is the composition of $B_0(z)=z^n$ with a M\"obius transformation preserving $\mathbb{T}$ and shifting $0$ to $\delta$:
\[ B_\delta(z)=  M_\delta(z^n), \textrm{ where we recall } M_\delta(z):= \frac{z+\delta}{1+\delta z}. \] 
The map $B_\delta$ is also a Blaschke product for non-real $\delta\in\mathbb{D}$, but in the course of proving Theorem \ref{main_thm2} we consider only real $\delta$. 
\end{rem}

\begin{notation} We use the notation
\[ \mathbb{D}(z,r):=\{\zeta\in\mathbb{C}: |\zeta-z|<r\}. \] 
\end{notation}

\begin{definition} We say a function $f$ is \emph{$\mathbb{R}$-symmetric} if and only if $f(\overline{z})=\overline{f(z)}$ for all $z$ in the domain of $f$. 
\end{definition}

\begin{construction}\label{qr_map_construction} We will now construct a quasiregular map $g: \mathbb{D}\rightarrow\mathbb{D}$ as follows. Fix $n>1$, and let $\delta\in\mathbb{R}$ satisfy
\begin{equation}\label{delta_assumption} |\delta|< \frac{n-1}{n+1}.
\end{equation}
The Blaschke product $B_\delta$ maps $\mathbb{D}(0,r)$ onto a hyperbolic disc centered at $\delta$; denote this hyperbolic disc by $D_\delta$. A brief calculation shows that for $r$ close enough to $1$ (depending on $\delta$), the assumption (\ref{delta_assumption}) implies 
\begin{equation}\label{inclusions} D_\delta\subset\mathbb{D}(0,r)\subset\mathbb{D}(0,\sqrt[n]{r}).
\end{equation} 
We henceforth fix such an $r$ such that (\ref{inclusions}) holds (we remark that for the purposes of proving Theorem \ref{main_thm2}, we need only consider $\delta$ arbitrarily close to $0$). Let 
\[ A:=\{z : r<|z|<\sqrt[n]{r}\}.\]
Note that $B_0: \sqrt[n]{r}\mathbb{T} \rightarrow r\mathbb{T}$ winds $n$ times around its image, and $B_\delta: r\mathbb{T} \rightarrow \partial D_\delta$ also winds $n$ times around its image. Thus (\ref{inclusions}) implies that there exists a smooth interpolation 
\begin{equation}\label{hdefn} h: A \rightarrow \{ z : |z| < r \textrm{ and } z\not\in D_\delta\}
\end{equation}
 so that $h|_{\sqrt[n]{r}\mathbb{T}}=B_0$ and $h|_{r\mathbb{T}}= B_\delta$. Moreover, since $B_0$ and $B_\delta$ are $\mathbb{R}$-symmetric, the interpolation $h$ may also be chosen to be $\mathbb{R}$-symmetric. Extend $h$ to a piecewise-defined mapping as follows:
\begin{equation}\label{g_defn} g(z):=\begin{cases} B_0(z)=z^n & \sqrt[n]{r} <|z|<1 \\ h(z) & z\in A, \\  B_\delta(z) & |z|<r.  \end{cases}
\end{equation}
\end{construction}

\begin{prop}\label{linearlyattpt} The map $g$ has a linearly attracting fixed point in $D_\delta\cap\mathbb{R}$. 
\end{prop}

\begin{proof} We have that $g(D_\delta\cap\mathbb{R})$ is compactly contained in $D_\delta\cap\mathbb{R}$, so $g$ has a fixed point $z_0$ within $D_\delta\cap\mathbb{R}$.  The Schwarz-Pick Lemma implies that $z_0$ must be attracting.  Moreover, we have $g'(z_0)\not=0$, so that $z_0$ is linearly attracting and not superattracting.  Indeed, the only critical point of $g$ is $0$, but $z_0 \neq 0$ since $z_0$ is fixed by $g$ and $g(0)=\delta\not=0$. 
\end{proof}

\begin{definition} A \emph{Beltrami coefficient} on a domain $\Omega\subseteq\Chat$ is a measurable function $\mu\in L^\infty(\Omega)$ satisfying $||\mu||_{L^\infty(\Omega)}<1$. Given a Beltrami coefficient $\mu$ on $\Omega$ and a quasiregular map $g: U \rightarrow \Omega$, the \emph{pullback} of $\mu$ under $g$ is defined as the following Beltrami coefficient on $U$:
\[ g^*\mu:=\frac{g_{\overline{z}}+\mu\circ g \cdot\overline{g_{z}}}{g_{z}+\mu\circ g \cdot g_{\overline{z}}}. \]
\end{definition}

\begin{definition}\label{mu_defn} We will now define a Beltrami coefficient $\mu$ on $\mathbb{D}$ using the map $g: \mathbb{D}\rightarrow\mathbb{D}$ of Construction \ref{qr_map_construction}. First, we set 
\[ \mu(z):=0 \textrm{  for } |z|<r.\]
Next, recalling that $h(A)\subset r\mathbb{D}$ by (\ref{hdefn}), we set 
\[ \mu(z):=h^*\mu(z) \textrm{ for } z\in A. \]
Lastly, we continue to pull back the definition of $\mu$ (already defined for $|z|<\sqrt[n]{r}$) under $g:=B_0$ in $\sqrt[n]{r}<|z|<1$, namely denoting $g^{d}:=g\circ...\circ g$ the $d^{\textrm{th}}$ iterate of $g$, we set:
\[ \mu(z):=(g^{d})^*\mu(z) \textrm{ for } z\in g^{-d}(A). \] 
\end{definition}

\noindent The key property of $\mu$ is the following.

\begin{prop} The Beltrami coefficient $\mu$ is $g$-invariant, namely $g^*\mu=\mu$. Moreover, $\mu$ is $\mathbb{R}$-symmetric. 
\end{prop}

\begin{proof} The fact that $\mu$ is $g$-invariant follows from Definition \ref{mu_defn}; indeed for $|z|<r$ we have that $g^{d}(z)\subset D_\delta$ for all $d$, and for $|z|>r$ the definition of $\mu$ was given so that the relation $g^*\mu=\mu$ is satisfied. The $\mathbb{R}$-symmetry of $\mu$ follows from the $\mathbb{R}$-symmetry of $h$ and hence $g$. 
\end{proof}

Let $p$ be as in Theorem \ref{Lukas}; for the purposes of proving Theorem \ref{main_thm2} it suffices to know there exists a real polynomial $p$ satisfying Theorem \ref{Lukas}, and in particular that $p$ is $\mathbb{R}$-symmetric. Denote by $\mathcal{A}$ the immediate basin of attraction for $\infty$ of $p$, namely $\mathcal{A}$ is the connected component of $\{z \in \Chat: p^d(z)\xrightarrow{d\rightarrow\infty}\infty\}$ containing $\infty$. There exists a Riemann (conformal) map $\phi: \mathcal{A} \rightarrow \mathbb{D}$, which we may normalize to be $\mathbb{R}$-symmetric and so that $\phi(\infty)=0$. The map $p: \mathcal{A} \rightarrow \mathcal{A}$ is proper, and hence the map $\phi \circ p\circ\phi^{-1}: \mathbb{D} \rightarrow \mathbb{D}$ is also proper and thus a Blaschke product. Since $\phi \circ p\circ\phi^{-1}$ is moreover of degree $n$ and has a unique fixed critical point at $0$, we conclude 
\begin{equation}\label{conjugacy'} \phi \circ p\circ\phi^{-1}(z) =\alpha z^n \textrm{ for } z\in\mathbb{D} \textrm{ and } |\alpha|=1. \end{equation}
Since $\phi$, $p$ are real, we conclude that $\alpha=\pm1$. In what follows, we will assume $\alpha=1$:
\begin{equation}\label{conjugacy} \phi \circ p\circ\phi^{-1}(z) =z^n \textrm{ for } z\in\mathbb{D}.
 \end{equation}
 
 \begin{rem} If $\alpha=-1$ in (\ref{conjugacy'}) one proceeds by adjusting the definition (\ref{g_defn}) of $g$ to interpolate between $-z^n$ and $(-z^n+\delta z)/(1-\delta z^n)$ rather than between $z^n$ and $(z^n+\delta z)/(1+\delta z^n)$. We assume $\alpha=1$ in what follows to simplify the exposition.
 \end{rem}



\begin{definition} We define a mapping $G: \Chat\rightarrow\Chat$ as follows. Recalling the definition of $g: \mathbb{D}\rightarrow\mathbb{D}$ from Construction \ref{qr_map_construction}, we set:
\begin{equation}G(z):=\begin{cases} p(z) & z\not\in \mathcal{A}, \\ \phi^{-1} \circ g \circ \phi & z\in \mathcal{A}.  \end{cases}
\end{equation}
\end{definition}

\begin{prop}\label{removabilityresult} The mapping $G: \Chat\rightarrow\Chat$ is quasiregular and $\mathbb{R}$-symmetric. 
\end{prop}
\begin{proof} The mapping $G$ is holomorphic outside of $\mathcal{A}$, and quasiregular in $\mathcal{A}$, and the  piecewise definition matches up on $\partial\mathcal{A}$ by (\ref{conjugacy}) and since $g(z)=z^n$ for $z\in\mathbb{T}$. Thus the conclusion will follow once we show that $\partial\mathcal{A}$ is removable for quasiregular mappings. To this end, we remark that all of the critical points of $p$ are either fixed or else lie in a $2$-cycle (see Section 2 of \cite{MR2358495}) and thus $p$ is hyperbolic; that is, all critical points of $p$ lie in attracting basins. Thus $\mathcal{A}$ is a John domain (see, for instance, Theorem 3.1 of \cite{MR1230383}), and then \cite{MR1315551} (either Theorem 1 or Corollary 2 of that paper) imply that $\partial \mathcal{A}$ is removable for quasiregular mappings. We direct the interested reader to \cite{MR3429163} for a broad overview of questions regarding removability properties of Julia sets.
\end{proof}

\begin{prop}\label{Glinattr} The map $G$ has a linearly attracting fixed point in $\mathcal{A}\cap\mathbb{R}$. 
\end{prop}

\begin{proof} By Proposition \ref{linearlyattpt}, $g$ has a linearly attracting fixed point in $D_\delta\cap\mathbb{R}$, call it $z_0$. Thus $\phi^{-1}(z_0)\in\mathcal{A}\cap\mathbb{R}$ is evidently a fixed point of $G$, and we readily check using the chain rule that
\begin{gather}\nonumber G'(\phi^{-1}(z_0))=(\phi^{-1})'(g\circ\phi(\phi^{-1}(z_0)))\cdot g'(\phi(\phi^{-1}(z_0)))\cdot \phi'(\phi^{-1}(z_0))=\\ \nonumber (\phi^{-1})'(z_0)\cdot g'(z_0)\cdot \phi'(\phi^{-1}(z_0)) = g'(z_0), \end{gather}
so that $\phi^{-1}(z_0)$ is a linearly attracting fixed point of $G$. 
\end{proof}

\begin{definition} We define a Beltrami coefficient $\nu$ on $\Chat$ as follows. Recall Definition \ref{mu_defn} of the Beltrami coefficient $\mu$ on $\mathbb{D}$. We set
\[ \nu:=\phi^*\mu \textrm{ on } \mathcal{A},  \] 
\[ \nu:=(G^d)^*(\nu|_{\mathcal{A}}) \textrm{ on } G^{-d}(\mathcal{A}), \textrm{ and }  \] 
$\nu:=0$ outside of the basin of attraction for $\infty$.
\end{definition}

\begin{prop}\label{Ginvariance} The Beltrami coefficient $\nu$ is $G$-invariant and $\mathbb{R}$-symmetric.
\end{prop}

\begin{proof} This follows from $g$-invariance of $\mu$ and $\mathbb{R}$-symmetry of $\mu$ and $\phi$. 
\end{proof}

\begin{notation} Let $z_0\in \mathcal{A}\cap\mathbb{R}$ denote the linearly attracting fixed point of $G$ from Proposition \ref{Glinattr}. 
\end{notation}

\begin{notation} We denote by $\Phi: \Chat \rightarrow\Chat$ the quasiconformal mapping satisfying $\Phi^*\nu\equiv0$ (the mapping $\Phi$ exists by the Measurable Riemann Mapping Theorem), and we normalize $\Phi$ so that $\Phi$ is $\mathbb{R}$-symmetric (this can be done since $\nu$ is $\mathbb{R}$-symmetric) and $\Phi(\infty)=\infty$, $\Phi(z_0)=z_0$. We define
\begin{equation}\label{r_defn} r:=\Phi\circ G\circ\Phi^{-1}: \Chat\rightarrow\Chat
\end{equation}
\end{notation}

\begin{prop}\label{risreal} The mapping $r: \Chat\rightarrow\Chat$ is holomorphic and hence a real rational mapping. 
\end{prop}

\begin{proof} This follows from Proposition \ref{Ginvariance}. 
\end{proof}

\begin{prop}\label{p_exist} There exists a degree $n$ polynomial $p(z)$ and $c\in\mathbb{C}$ so that $r(z)=c+1/p(z)$.
\end{prop}

\begin{proof} We have that 
\[ r(\infty):=\Phi\circ G\circ\Phi^{-1}(\infty)=\Phi \circ G(\infty) = \Phi \circ \phi^{-1} \circ g \circ \phi (\infty) = \Phi \circ \phi^{-1} \circ g (0) =  \Phi\circ \phi^{-1}(\delta) \in\mathbb{C}. \]
Similarly, one checks that $\infty$ is a degree $n-1$ critical point of $r$ (since $0$ is a degree $n-1$ critical point of $g$). Thus the conclusion follows from Lemma \ref{form_of_r_lem}.
\end{proof} 

\begin{notation} Let $z_1$, ..., $z_{n-1}$ denote the $n-1$ finite, fixed critical points of $\overline{p}$.
\end{notation}

\begin{prop}\label{cfpofr} The points $\Phi(z_1)$, ..., $\Phi(z_{n-1})$ are finite, fixed critical points of $\overline{r}$.
\end{prop}

\begin{proof} Let $1\leq j \leq n-1$. Since $z_j$ is a fixed point of $\overline{p}$, we have 
\begin{equation}\label{zjreltn} p(z_j)=\overline{z_j},
\end{equation} 
and to show $\Phi(z_j)$ is a fixed point of $\overline{r}$ it suffices to show $r(\Phi(z_j))=\overline{\Phi(z_j)}$. Since $z_j\not\in\mathcal{A}$, we see that
\[ r(\Phi(z_j)):=\Phi\circ G\circ\Phi^{-1}(\Phi(z_j))= \Phi\circ p(z_j), \] 
and by $\mathbb{R}$-symmetry of $\Phi$ and (\ref{zjreltn}), we conclude
\[ \Phi\circ p(z_j) = \Phi(\overline{z_j})=\overline{\Phi(z_j)}, \] 
so that $r(\Phi(z_j))=\overline{\Phi(z_j)}$, as needed.

\end{proof}

\begin{prop}\label{lafpofr} The point $z_0$ is a linearly attracting fixed point of $\overline{r}$.
\end{prop}

\begin{proof} Since $\Phi$ is a conjugacy between $r$ and $G$ and $z_0$ is a fixed point of $G$, it is evident that $\Phi(z_0)$ is a fixed point of $r$, and we normalized $\Phi$ so that $\Phi(z_0)=z_0$. Moreover, we check that:
\[ r'(z_0)=\Phi'(z_0)\cdot G'(z_0)\cdot (\Phi^{-1})'(z_0)=G'(z_0), \]
so since $z_0$ is linearly attracting for $G$ the same is true of $r$. Since $z_0$ is real, $z_0$ is also a linearly attracting fixed point of $\overline{r}$.
\end{proof}

\noindent We have now proven Theorem \ref{main_thm2}; indeed first we defined in (\ref{r_defn}) a mapping $r$ and proved in Proposition \ref{p_exist} that there exists a polynomial $p$ and $c\in\mathbb{C}$ so that $r(z)\equiv c+1/p(z)$. Then we demonstrated in Propositions \ref{cfpofr}, \ref{lafpofr} the existence of $n$ attracting fixed points $z_0$, ..., $z_{n-1}$ of $\overline{r(z)}=\overline{c}+1/\overline{p(z)}$.

\section{Nonexistence of post-critically finite examples}

We now prove the fact, discussed in Remark \ref{no_pcf_rem}, that there are no post-critically finite maps satisfying Theorem \ref{main_thm}.

\begin{prop}\label{no_pcf} Suppose the anti-rational map $\overline{r}(z):=\overline{c}+1/\overline{p(z)}$ has $n$ attracting fixed points. Then the orbit of $\infty$ is infinite; in particular $\overline{r}$ is not post-critically finite. 
\end{prop}

\begin{proof} Suppose $\overline{r}(z):=\overline{c}+1/\overline{p(z)}$ has $n$ attracting fixed points. As already discussed, the map $\overline{r}$ has $2n-2$ critical points (counting multiplicity), and $\infty$ is a critical point of degree $n-1$. This means that, not counting multiplicity, $\overline{r}$ has at most $n$ critical points. Since each attracting fixed point attracts a critical point, this means each critical point of $\overline{r}$ lies in the basin of attraction of an attracting fixed point. Consider the critical point $\infty$, and call $z_0$ the attracting fixed point so that 
\[\overline{r}^d(\infty)\xrightarrow{d\rightarrow\infty}z_0.\]
In order to use the more readily-available theory of holomorphic dynamics (in place of anti-holomorphic dynamics), let us consider the second iterate $s:=\overline{r}^2$: now $s$ is a holomorphic mapping, and we still have 
\[ s^d(\infty)\xrightarrow{d\rightarrow\infty}z_0. \] 
According to K\oe nig's Linearization Theorem (see for instance Theorem 8.2 and Lemma 8.5 of \cite{MilnorCDBook}), there exists a neighborhood $U$ of $z_0$, $t<1$ and a conformal mapping $\phi: U\rightarrow t\mathbb{D}$ so that
\[ \phi\circ s\circ\phi^{-1}(z)=\lambda\cdot z \textrm{ for } z\in \overline{t\mathbb{D}} \textrm{ and } \lambda:=s'(z_0), \] 
and moreover $\phi^{-1}(t\mathbb{T})$ contains a critical point of $s$; denote this critical point by $\zeta$. Since the orbit (under $z\mapsto\lambda z$) of any point lying on $t\mathbb{T}$ is infinite, it follows that the orbit of the critical point $\zeta$ is infinite. The chain rule implies that since $\zeta$ is a critical point of $s$, then $\zeta$ is either a critical point of $\overline{r}$ or else a critical value of $\overline{r}$. Since the non-infinite critical points and values of $\overline{r}$ are contained in separate basins of attraction, it follows that $\zeta$ is either $\infty$ or $r(\infty)$. In either case, we see that the orbit (under $\overline{r}$) of $\infty$ is infinite. This means, by definition, that $\overline{r}$ is not post-critically finite.
\end{proof}

\subsection*{Acknowledgements.} The authors would like to thank Sabyasachi Mukherjee and Malik Younsi for pointing out the references used in the proof of Proposition \ref{removabilityresult}.



\bibliographystyle{alpha}

\end{document}